\documentclass[a4paper,10pt]{article}
\usepackage{amsmath}
\usepackage{amsthm}
\usepackage{mathrsfs}
\usepackage{latexsym}
\usepackage{amssymb}
\usepackage{amscd}
\usepackage[dvips]{graphics}
\usepackage[all,cmtip]{xy}
\usepackage{enumerate}
\usepackage[colorlinks=true]{hyperref}
\usepackage{tikz-cd}
\usepackage{fullpage}
\usepackage{enumitem}
\usepackage{mathabx}
\usepackage{soul}
\usepackage{xcolor}

\usepackage{palatino}
\usepackage{authblk}

\hypersetup{colorlinks=true,linkcolor=[rgb]{0.5, 0.0, 0},citecolor=[rgb]{0.5, 0.0, 0.5}, filecolor=magenta,urlcolor=blue}

\numberwithin{equation}{section}
\newtheorem{theorem}[equation]{Theorem}
\newtheorem{proposition}[equation]{Proposition}
\newtheorem{lemma}[equation]{Lemma}

\newtheorem{corollary}[equation]{Corollary}

\theoremstyle{remark}
\newtheorem{rmk}[equation]{Remark}
\theoremstyle{definition}

\newtheorem{example}[equation]{Example}

\newcommand{\gd}{\operatorname{gldim}}
\newcommand{\pd}{\operatorname{pd}}
\newcommand{\spec}{\operatorname{Spec}}

\newcommand{\gr}{\operatorname{gr}}
\newcommand{\Max}{\operatorname{Max}}
\newcommand{\fd}{\operatorname{fd}}
\newcommand{\Tor}{\operatorname{Tor}}
\newcommand{\Ext}{\operatorname{Ext}}
\newcommand{\cdim}{\operatorname{cdim}}

\newcommand{\HH}{\operatorname{HH}}
\newcommand{\rk}{\operatorname{rank}}

\title{A relative homology criterion of smoothness}

\author{Kostiantyn Iusenko, Eduardo do Nascimento Marcos, Victor do Valle Pretti}
\affil{Instituto de Matem\'{a}tica e Estat\'{i}stica, Univ. de S\~{a}o Paulo, S\~{a}o Paulo, SP, Brazil}

\date{}

\begin{document}

\maketitle

\begin{abstract}
We investigate the relationship between smoothness and the relative global dimension of a ring extension. We prove that a smooth commutative algebra $A$ over $B$ has finite relative global dimension $\gd(A,B)$. Conversely, under a mild condition on $B$, the finiteness of $\gd(A,B)$ implies that the map $B \to A$ is smooth. We also relate the relative global dimension to the usual global dimension of the fibers of $B \to A$, and establish a formula for the relative global dimension of tensor products of extensions. Finally, we present examples and an alternative characterization of smoothness in terms of relative Hochschild homology.

\medskip

\noindent\textbf{Keywords:} relative homology, smooth morphisms, commutative rings, global dimension.
\end{abstract}

\section{Introduction}

Smoothness is a fundamental concept in algebraic geometry, providing a key link between geometric and algebraic properties of varieties. A fundamental result due to Auslander-Buchsbaum and Serre (see \cite{AB, Serre}) states that if $V$ is an affine algebraic variety over a perfect field $k$ with coordinate ring $A$, then the global dimension of $A$ is finite if and only if $V$ is smooth. Instead of a map $k\to A$, one can consider a more general case in which $k$ is replaced by a commutative ring $B$. A well-established criterion of smoothness in this case (see \cite{Lod}) has a number of homological characterizations (see, for instance, \cite{Rodi,Av-Iyen} and references therein).

In this manuscript, we provide a characterization of smoothness via relative global homology developed by Hochschild in \cite{Hoc}, specifically focusing on the relative global dimension $\gd(A,B)$. Surprisingly, there is little literature on relative homological algebra for associative algebras. In contrast, a vast body of work exists for relative homological algebra in the context of  representations of  finite groups — results that have had a profound impact on modular representation theory and block theory of finite groups (see, for instance, \cite{Lin}). However, there has been a recent surge in interest in the application of relative homological algebra to attack certain homological conjectures for finite-dimensional algebras, such as the finitistic dimension conjecture and Han's conjecture. Interested readers can explore these recent developments and related work in \cite{XiXu, CLMS21, IM21} and references therein.

The manuscript is structured as follows. In Section~\ref{SecPrem}, we review basic notions in relative homological algebra, together with an overview of graded and filtered algebras and modules. Section~\ref{SecSmRel} is devoted to showing that if $A$ is a smooth commutative $B$-algebra, then its relative global dimension $\gd(A,B)$ is finite. In Section~\ref{SecCohom}, we study the converse direction. We establish a sufficient condition on $B$ under which the finiteness of $\gd(A,B)$ implies that the morphism $B \to A$ is smooth. Combining both results, we obtain our main characterization (Corollary~\ref{CorCriteriaPerfect}):  
if $k$ is a perfect field, $B$ a finitely generated $k$-algebra, and $A$ a flat $B$-algebra,  essentially of finite type, then  
\[
    B \to A \text{ is smooth } \quad \Longleftrightarrow \quad \gd(A,B) < \infty.
\]
This characterization also leads to further consequences. In particular, we deduce Corollary~\ref{cor-Fibers-2}, which—under the assumptions of Corollary~\ref{CorCriteriaPerfect}—relates the relative global dimension $\gd(A,B)$ to the global dimensions of the fibers of $B \to A$:
\[
    \gd(A,B)
      = \sup_{\mathfrak{q} \in \spec B} \gd(A \otimes_B k(\mathfrak{q}))
      = \sup_{\mathfrak{m} \in \Max B} \gd(A \otimes_B k(\mathfrak{m})).
\]
Using this result, we prove Proposition~\ref{prop-Gldim-tensor}, which provides a relative analogue of~\cite[Theorem~16]{Aus55} (for the case $B=k=D$):  
for two extensions $B \subseteq A$ and $D \subseteq C$ satisfying the assumptions of Corollary~\ref{CorCriteriaPerfect}, one has
\[
    \gd(A \otimes_k C, B \otimes_k D)
      = \gd(A, B) + \gd(C, D).
\]
We believe that the last two results are of independent interest. Finally, Section~\ref{SecExamp} concludes the paper with examples and an additional characterization of smoothness in terms of relative Hochschild homology, in the spirit of the Semi-Rigidity Theorem of~\cite{Av-Iyen} and \cite{Av-VP,BACH} (for the case $B=k$).

\textbf{Acknowledgements.} The first-named author thanks Changchang Xi for stimulating discussions on relative homological dimension during the early stages of this work. We also thank Daniel Levcovitz for helpful discussions, and Andrea Solotar and John MacQuarrie for their valuable feedback on a draft of this manuscript. K.I. and E.M. were partially supported by FAPESP grant 2018/23690-6 and by CNPq Universal Grant 405540/2023-0, K.I. was also partially supported by FAPEMIG Project APQ-03491-25, V.P. by FAPESP Scholarship 2022/12883-3, and E.M. by CNPq Research Grant 310651/2022-0. Finally, we thank the anonymous referee for numerous valuable comments that helped improve this manuscript.

\section{Preliminaries} \label{SecPrem}

In this section, we introduce some notation and review basic definitions and results needed for proving our main theorems. We work within the category of associative commutative rings with identity. Given a ring $A$, we denote by $A\mathrm{-Mod}$ the category of $A$-modules. Throughout the text, for a ring map $B\to A$, we denote by $A^e$ the ring $A\otimes_B A$ and treat $A$ as an $A^e$-module via the canonical surjective map $\mu:A\otimes_B A\to A$, $a\otimes a'\mapsto aa'$. 
For any ring $A$, we denote by $\spec A$ the topological space of prime ideals equipped with the Zariski topology, and by $\Max A$ its set of closed points, i.e., the maximal ideals of $A$.

\subsection{Relative (co)homology}

In order to describe relative homological algebra (cf. \cite{Hoc}), we first recall the notion of relatively projective modules. Given a homomorphism between rings $B \to A$, an $A$-module $M$ is called \textit{relatively $B$-projective}, or $(A,B)$-\textit{projective}, if it satisfies one, and therefore all, of the following equivalent conditions:
\begin{itemize}
    \item[(1)] the multiplication map $\mu_M : A \otimes_B M \to  M$ is a split epimorphism of $A$-modules, where $A\otimes_B M$ carries the natural left $A$-module structure;
    \item[(2)] $M$ is isomorphic, as an $A$-module, to a direct summand of the induced module $A\otimes_B V$, for $V$ some $B$-module;
    \item[(3)]  if ever an $A$-module homomorphism onto $M$ splits as a $B$-module homomorphism, then it splits as an $A$-module homomorphism.
\end{itemize}

An exact sequence of $A$-module homomorphisms:
$$\cdots\to  M_{n+1}\xrightarrow{f_{n+1}} M_{n}\xrightarrow{f_{n}} M_{n-1}\to 0$$
is called $(A,B)$-exact if, for each $i \ge n$, the kernel of $f_i$ is a direct $B$-module summand of $M_i$ (cf. \cite[Section 1]{Hoc}). One may check that a sequence of morphisms
$\{f_i:M_i\to M_{i-1}\ |\ i \ge n\}$ is $(A,B)$-exact if, and only if,
\begin{itemize}
\item[(1)] $f_i \circ f_{i+1} = 0$ for all $i > n$;
\item[(2)] there exists a contracting $B$-homotopy: that is, a sequence of $B$-module homomorphisms $h_i: M_i \to M_{i+1}$, $(i \geq n-1)$ such that 
$f_{i+1}h_i + h_{i-1}f_i$ is the identity map on
$M_i$.
\end{itemize}

One may now develop the concepts of relative projective dimension and relative global dimension. Given an $A$-module $M$, we define the \textit{relative projective} dimension of $M$ to be the minimal number $n$, denoted by $\pd_{(A,B)} M$, such that there is an $(A,B)$-exact sequence
(called $(A,B)$-\textit{projective resolution} of $M$)
$$
0 \rightarrow P_{n} \rightarrow \cdots \rightarrow P_{1} \rightarrow P_{0} \rightarrow M \rightarrow 0.
$$
where the $P_i$ are $(A,B)$-projectives. If such an exact sequence does not exist, the relative projective dimension of $M$ is infinite. 
This definition is equivalent to the corresponding definition from \cite[Section 2]{XiXu}.  Having relative projective resolutions, the relative derived functors $\Tor^{(A,B)}_n$ and $\Ext_{(A,B)}^n$ can be defined, and we refer to \cite{Hoc} for the details.

\begin{rmk}
Note, that the class of $(A,B)$-exact sequences defines an exact structure on $A\mathrm{-Mod}$: namely, a sequence is $(A,B)$-exact if and only if it becomes split exact upon restriction to $B$. The $(A,B)$-projective modules are precisely the projective objects with respect to this exact structure, and the corresponding $\Ext$-groups coincide with the $\Ext$-groups of this exact category. For more information on the topic, see \cite{Buhler}.
\end{rmk}

\begin{rmk}\label{SplicingKernels}
Given an $M\in A\mathrm{-Mod}$, one produces  the \emph{standard $(A,B)$-projective resolution} (check \cite[Section 2]{Hoc}) by splicing the short $(A,B)$-exact sequences 
$$ 0\rightarrow \ker{\mu_n} \rightarrow A\otimes_B K_{n} \xrightarrow{\mu_n}  K_{n} \rightarrow 0$$ 
where $A\otimes_B K_n$ is $(A,B)$-projective, and $K_1=M$, $K_{i+1}=\ker \mu_i$ for $i\ge1$. 
\end{rmk}

The \textit{relative global dimension} $\gd(A,B)$ of the extension $B \to A$ is defined as 
$$
    \gd(A,B)=\sup\{\pd_{(A,B)} M \ | \ M\in A\mathrm{-Mod}\}.
$$
The \textit{relative global cohomological dimension} is defined as 
$$\cdim(A,B)= \sup \{ n\ | \Ext^n_{(R,S)}(A,-)\neq 0\},$$
in which $R=A\otimes_B A$, $S=B\otimes_B A$ with the natural map $S\rightarrow R$ and $A$ with the natural structure of $R$-module. 
These dimensions are related as follows.

\begin{rmk}\label{remCdim}
    From \cite[Corollary 1]{Hoc} we have that $\cdim(A,B)$ can also be calculated as $\gd(A\otimes_B A, B\otimes_B A)$. Moreover, it is clear from the definition that \[\cdim(A,B)=\pd_{(A\otimes_B A, B \otimes_B A)} A.\] Therefore applying $-\otimes_A M$ to an $(A\otimes_B A, B \otimes_B A)$-resolution of $A$ for each $A$-module $M$ one gets that (see also \cite[Corollary 1]{Hoc}):
\begin{equation}\label{Lem-Hos-Upb}
\gd(A,B)\leq \cdim(A,B). 
\end{equation}
\end{rmk}

\subsection{Graded and filtered algebras}

Recall the basic definition about filtered and graded algebras (cf. \cite[Chapters 7, 12]{NCNR} and \cite{Grad-R}). In this case, the focus will be on descending filtrations. A \textit{filtered} $A$-algebra $R$ is defined as an $A$-algebra satisfying the condition $$R=\bigcup_{i \in \mathbb{N}} R_i,$$ where $R_i$ are $R$-ideals, subject to the following properties:
\begin{itemize}
\item[(1)] $R_i R_j \subseteq R_{i+j}$;
\item[(2)] $R_{i+1} \subseteq R_i$.
\end{itemize}
In particular, $R=R_0$. Furthermore, if the $R_i$ are flat $A$-modules for every $i$, then $R$ is referred to as a \textit{flat filtered} $A$-algebra. Given a filtered $A$-algebra $R$, a \textit{filtration} of an $R$-module $M$ is defined as a collection of $R$-modules $M_i$, satisfying:
\begin{itemize}
\item[(1)] $M = \bigcup_{i \in \mathbb{N}} M_i$;
\item[(2)] $M_{i+1}\subseteq M_{i}$
\item[(3)] $R_i M_j \subseteq M_{i+j}$.
\end{itemize}
Consequently, it is necessary that $M=M_0$. A homomorphism $\phi:M \rightarrow N$ between filtered $R$-modules $M=\bigcup_{i\in \mathbb{N}} M_i$ and $N=\bigcup_{i\in \mathbb{N}} N_i$ is said to be a \emph{filtered morphism} if $\phi(M_i) \subseteq N_i$. 

\begin{example}\label{gradedExam}
Consider a homomorphism $B \rightarrow A$. Let $R=A^{e}$ and $J = \ker(\mu : A^e \rightarrow A)$. Then $R$ is a filtered $A$-algebra with $R_i = J^i$ and $R_0=A^{e}$. Assuming that $B \rightarrow A$ is a flat ring map and $J^i/J^{i+1}$ are flat $A$-modules, it follows that $R$ is a flat filtered $A$-algebra.
\end{example}

In conjunction with the concept of filtered algebras and modules, we  give 
the definitions of graded algebras and modules. In our context, we focus specifically on non-negatively graded algebras.

A \textit{graded} $A$-algebra $R$ is defined as an $A$-algebra $R=\bigoplus_{i \in \mathbb{N}}R_i$, where each $R_i$ is an additive subgroup of $R$, satisfying the properties $R_i R_j \subseteq R_{i+j}$ and $R_0=A$. Similarly, a \textit{graded} $R$-module $M$ is a module decomposed as $M=\bigoplus_{i \in \mathbb{N}}M_i$, where $M_i$ are additive subgroups of $M$ with $R_i M_j \subseteq M_{i+j}$.

The elements in $R_i$ are referred to as \emph{homogeneous of degree $i$}. Additionally, an ideal $I$ of a graded $A$-algebra $R$ is called a \emph{homogeneous ideal} if it can be generated by homogeneous elements. Furthermore, $R$ is called \emph{graded local} if it possesses only one maximal homogeneous ideal. We denote the homogeneous ideal $\bigoplus_{i > 0} R_i$ by $R_+$.

\begin{rmk}
There exists a natural construction to pass from a filtered ring to a graded ring, achieved through the grading functor $\gr$. Given a filtered $A$-algebra $R$ such that $A=R_0/R_1$, one constructs the \emph{associated graded ring} $\gr(R) = \bigoplus_{i \in \mathbb{N}} R_i/R_{i+1}$. This is naturally a graded $A$-algebra. Moreover, if $A$ is a local ring with maximal ideal $\mathfrak m$, then $\gr(R)$ becomes a graded local $A$-algebra with homogeneous maximal ideal $\mathfrak m \oplus \gr(R)_+$.

For any filtered $R$-module $M$, we can also construct $\gr(M) = \bigoplus_{i \in \mathbb{N}} M_i/M_{i+1}$, equipped with a $\gr(R)$-module structure. If $\phi: M \rightarrow N$ is a homomorphism of filtered $R$-modules, then there exists a $\gr(R)$-homomorphism $\gr(\phi): \gr(M) \rightarrow \gr(N)$.
\end{rmk}

\begin{lemma}\label{Lem-Grd-Prod}
Given a flat filtered $A$-algebra $R$ with $R=\bigcup_{i \in \mathbb{N}} R_i$ such that  $A=R_0/R_1$ and  $R_i/R_{i+1}$ are flat $A$-modules, for each $i$, one has
\[\gr(R \otimes_A M)=\gr(R)\otimes_A \gr(M)\]
for any filtered $R$-module $M$.
\end{lemma}

\begin{proof}
Since $R_i/R_{i+1}$ is flat over $A$, the inclusions $R_{i+1}\subseteq R_{i}$ remain injective after tensoring, so the filtration on $R$ behaves well with respect to tensor products. 
Moreover, $R_kR_i\subseteq R_{k+i}$ implies
$$
R_k(R_i\otimes_A M_j)\subseteq R_{k+i}\otimes_A M_j,
$$
and since each $R_i$ is a flat $A$-module, the inclusion $R_i\subseteq R$ remains injective after tensoring with $M$. Thus, $R_i\otimes_A M$ may be regarded as a submodule of $R\otimes_A M$, and hence so may $R_i\otimes_A M_j$.
Therefore
$$
(R\otimes_A M)_n=\sum_{i+j=n} R_i\otimes_A M_j
$$
defines a natural filtration of $R\otimes_A M$. Furthermore, we have the natural exact sequence:
$$ 0 \rightarrow  \sum_{i+j=n+1} R_i \otimes_A M_j \rightarrow \sum_{i+j=n} R_i \otimes_A M_j \rightarrow \underset{i+j=n}{\bigoplus} R_i/R_{i+1}\otimes_A M_j/M_{j+1} \rightarrow 0,$$
which concludes the proof.
\end{proof}

\begin{rmk}
    In what follows, we utilize two types of localization functors with respect to a given prime ideal in a graded $A$-algebra. Let $\mathfrak{p}$ be a prime ideal in the graded $A$-algebra $R$.  Define two localizations: $S_{\mathfrak{p}}=R\setminus \mathfrak{p}$ and $S_{(\mathfrak{p})}=h(R)\setminus \mathfrak{p}$, where $h(R)$ denotes the set of homogeneous elements in $R$, and consider the $A$-algebras $R_{\mathfrak{p}}=S_{\mathfrak{p}}^{-1}R$ and $R_{(\mathfrak{p})}=S_{(\mathfrak{p})}^{-1}R$. Then $R_{\mathfrak{p}}$ is a local ring with maximal ideal $\mathfrak{p}R_{\mathfrak{p}}$, while $R_{(\mathfrak{p})}$ is a local graded ring with homogeneous maximal ideal $(\mathfrak{p})_gR_{(\mathfrak{p})}$, where $(\mathfrak{p})_g$ denotes the homogenization of the ideal $\mathfrak{p}$, i.e. the homogeneous ideal contained in $\mathfrak{p}$ such that no other homogeneous ideal contained in $\mathfrak{p}$ contains it. For further details, refer to \cite[Chp B, III-1 to 3]{Grad-R}.

\end{rmk}

\section{Relative global dimension and smoothness}

\subsection{Relative global dimension of smooth algebras} \label{SecSmRel}

 The notion of smooth algebras is well-established in the literature, see  \cite[Appendix E]{Lod} and references therein. Given a noetherian ring $B$ and a $B$-algebra $A$ of essentially finite type, $A$ is called \textit{smooth} if
     \begin{itemize}
         \item[(1)] $A$ is a flat $B$-module,  when considered as a module via the ring map $B\rightarrow A$;
         \item[(2)] and the ideal $\ker(\mu: A^{e} \rightarrow A)$ is a locally complete intersection.
     \end{itemize}
Recall, that an ideal $J$ of a ring $A$ is said to be a \textit{locally complete intersection} if, for every maximal ideal $\mathfrak{m}$ of $A$, the ideal $J_{\mathfrak{m}}$ is generated by an $A_{\mathfrak{m}}$-regular sequence.

Throughout this section, denote by $J$ the kernel of the multiplication map $\mu: A^{e} \rightarrow A$, and by $\Omega_{A|B}$ the $A$-bimodule $J/J^2$. The  $A$-bimodules $J^i/J^{i+1}$ carry a natural structure of $A$-modules, by choosing either one of the natural morphisms $A\rightarrow A\otimes_B A$ induced by tensoring on either side the ring map $B\rightarrow A$. This is independent of the choice, since $J$ is generated by elements of the form $(a\otimes 1-1\otimes a)$, for $a\in A$. Furthermore, we will consider $A^e$ (and its localizations) as a filtered $A$-algebra with the filtration described in Example \ref{gradedExam} (and its respective localizations).

\begin{rmk}
Observe that there are two types of maximal ideals in $A^{e}$: A first type of the form $\mu^{-1}(\mathfrak m)$ for some maximal ideal $\mathfrak m$ of $A$. Furthermore, the localization of $A$, as an $A^{e}$-module, at such an ideal is equal to $A_{\mathfrak m}$. The second type consists of those $\mathfrak n$ where the localization of $A$ as an $A^{e}$-module is trivial, that is, $\mu(\mathfrak n)=A$. Moreover, localizing $J^k/J^{k+1}$ as an $A^{e}$-module at a maximal ideal of the form $\mu^{-1}(\mathfrak m)$ is the same as localizing it as an $A$-module at $\mathfrak m$.
\end{rmk}

\begin{rmk}\label{rmk-Smo-Local}
One consequence of the definition of smoothness above, which is sometimes used as part of the definition of smoothness for noetherian $B$-algebras, is that the $A$-module $\Omega_{A|B}$ is finitely generated and projective. Another consequence is that $\gr(A^{e}_{\mu^{-1}(\mathfrak{m})})$ is isomorphic to $A_{\mathfrak{m}}[x_1,...,x_n]$, where $\mathfrak{m}$ is a maximal ideal of $A$ and the $x_i$ are determined by the regular sequence generating the ideal $J_{\mu^{-1}(\mathfrak{m})}$.
\end{rmk}

\begin{proposition}\label{Pro-grdim}
 A smooth noetherian $B$-algebra $A$ satisfies the following conditions
\begin{itemize}
\item[$(i)$] $\pd_{\gr(A^{e})}A$ is finite;
\item[$(ii)$] $\gr(A^{e})$ is a projective $A$-module.
\end{itemize}
\end{proposition}

\begin{proof}

$(i)$ A homogeneous maximal ideal $\mathfrak{m}$ in $$\gr(A^{e})=A^{e}/J\oplus J/J^2 \oplus \dots \oplus J^{i}/J^{i+1}\oplus \dots$$ always has the form 
\begin{equation}\label{eq-maxAemaxA}
    \mathfrak{m}=\mathfrak{n}  \oplus J/J^2 \oplus \dots \oplus J^{i}/J^{i+1}\oplus \dots,
\end{equation} where $\mathfrak{n}$ is a maximal ideal in $A^e/J=A$. Moreover, since $A$ is smooth, Remark \ref{rmk-Smo-Local} implies that the $A$-modules $J^i/J^{i+1}$ are locally free and generated by the $i$-th degree homogeneous polynomials with variables as the representatives of regular sequences generating $J$.

Now, observe the compatibility of the following localizations, for a given homogeneous maximal ideal $\mathfrak{m}$ of $\gr(A^{e})$ and $\mathfrak{n}$ its corresponding maximal ideal of $A$ as in \eqref{eq-maxAemaxA}: \[ (\gr(A^{e}))_\mathfrak{m}=((\gr(A^{e}))_{(\mathfrak{m})})_{\mathfrak{m}^e},\] where $\mathfrak{m}^e$ is the extension of $\mathfrak{m}$ in $\gr(A^{e})_{(\mathfrak{m})}$. From our description of $\gr(A^{e})_{(\mathfrak{m})}$ and \cite[Proposition 1.5.15]{CMR}, we have $\pd_{\gr(A^{e})_\mathfrak{m}}A_{\mu(\mathfrak{m})} \leq n$. Note that the number of variables in the polynomial ring $\gr(A^{e})_{(\mathfrak{m})}$ is equal to the rank of $J_{\mu^{-1}(\mathfrak{n})}/J_{\mu^{-1}(\mathfrak{n})}^2$. Since $\Omega_{A|B}$ is a finite and projective $A$-module, its rank is locally constant \cite[\href{https://stacks.math.columbia.edu/tag/00NV}{Tag 00NV}]{stacks-project}, and therefore \begin{equation}\label{Eq-cota-superior}
    \pd_{\gr(A^{e})} A=\sup \{\pd_{\gr(A^{e})_\mathfrak{m}} A_{\mu(\mathfrak{m})} \} \le l,
\end{equation} with the supremum taken for $\mathfrak{m}$ being in the set of graded maximal ideals of $\gr(A^{e})$ and $l$ the maximal rank of the free $A_{\mu(\mathfrak{m})}$-modules $(\Omega_{A|B})_{\mu(\mathfrak{m})}$.

$(ii)$ This follows from Remark \ref{rmk-Smo-Local}. As  $\gr(A^e)$ is an $A$-direct sum of the $A$-modules $J^i/J^{i+1}$, which, when localized at each maximal ideal $\mathfrak{n}$ of $A$, become the $A_\mathfrak{n}$-modules of the form $J_{\mu^{-1}(\mathfrak{n})}^i/J_{\mu^{-1}(\mathfrak{n})}^{i+1}$. These modules are projective $A_\mathfrak{n}$-modules for every $i$, and moreover, they are free by the isomorphism given in \cite[Ex.17.16]{Eis}. Therefore, each $J^i/J^{i+1}$ is $A$-projective and we get the claim.
\end{proof}

\begin{rmk}
 An immediate consequence of the proposition is that, for commutative rings, being a smooth noetherian $B$-algebra implies \emph{homological smoothness}; that is, $A$ admits a finite projective resolution by finitely generated $A^{e}$-modules. Moreover, if $A^{e}$ were projective as an $A$-module, then the proof of the next theorem could be considerably simplified by applying \cite[Corollary 2]{Rodi}, avoiding the need to pass to the associated graded ring. However, Example \ref{Ex-Aenonproj} shows that this projectivity assumption does not hold in general.
\end{rmk}

\begin{theorem}\label{Teo-Smo-Ida}
The relative global dimension $\gd(A,B)$ is finite for a smooth noetherian $B$-algebra $A$. Moreover, it is bounded by the projective dimension of $A$ as a $\gr(A^{e})$-module.
\end{theorem}

\begin{proof}
We begin the proof by constructing a standard $(A^{e},A)$-projective resolution of $A$, where each term of the resolution is of the form $A^{e}\otimes_A V$ for some $A$-module $V$. Note that the multiplication map $\mu : A^{e} \rightarrow A$ induces the $(A^{e},A)$-exact sequence
\begin{equation}\label{Eq-Prin-R-A}
0 \rightarrow J \rightarrow A^{e} \xrightarrow{\mu} A \rightarrow 0,
\end{equation}
where $J = \ker(\mu)$. Then proceeding as in Remark \ref{SplicingKernels}, we obtain the long $(A^{e},A)$-exact sequence by splicing the corresponding sequence of short $(A^{e},A)$-exact sequences:

\centerline{
\xymatrix{
												&&&&K_1 \ar@{^{(}->}[rd]		\\		
	...	\ar[rr] \ar[rd]		&&A^{e}\otimes_A K_{n-1}  \ar[r]	 &...	\ar[rr]  \ar@{.>}[ru]	 &&A^{e}\otimes_A J	\ar@{->>}[rd] \ar[rr] 	&&A^{e}	\ar@{->>}[r]	&A \\
		        &K_n  \ar@{^{(}->}[ru]								&&&&&J  \ar@{^{(}->}[ru]
}}
Observe that we can view this sequence as a resolution of filtered $A$-modules, where we induce the filtration inductively from the filtration of \eqref{Eq-Prin-R-A} and by taking the product filtration, as in the proof of Lemma \ref{Lem-Grd-Prod}.

Now, we consider the graded resolution associated with this filtered resolution. Since $\gr(A^{e})$ is a projective $A$-module and by Lemma \ref{Lem-Grd-Prod}, we conclude that this is also a $\gr(A^e)$-projective resolution of $A$. Using the fact that $A$ has finite projective dimension as a $\gr(A^{e})$-module, there exists a positive integer $n$ such that truncating the resolution at degree $n$ yields a projective $\gr(A^{e})$-module $\gr K_n$. Applying \cite[Cor 12.2.9]{NCNR}, one can verify that $\gr K_n$ is isomorphic to $\gr{A^{e}} \otimes_A Q$, where $Q$ is a graded projective $A$-module, considering $A$ as a trivially graded ring. Using the same argument as in the proof of \cite[Theorem 12.3.4]{NCNR}, we conclude that the corresponding term $K_n$ is $(A^{e},A)$-projective, as it is isomorphic to $A^{e} \otimes_A Q$. Therefore, $\pd_{(A^{e},A)} A$ is finite, and by \eqref{Lem-Hos-Upb}, $\gd(A,B)$ is also finite.
\end{proof} 
\begin{rmk}\label{remSmoothCdim} As follows from the proof of the above theorem, for a smooth noetherian $B$-algebra $A$ we have 
$\pd_{(A^{e},A)} A$ finite. Hence, by Remark~\ref{remCdim}, 
$\cdim(A,B)$ is finite too.
\end{rmk}

As a consequence, we obtain a bound for the relative global dimension of a smooth extension of rings, expressed in terms of the global dimension of the fibers of the associated morphism of affine varieties. For a ring $R$ and a prime ideal $\mathfrak{p} \subseteq R$, we denote by $k(\mathfrak{p})$ its \emph{residue field}, defined as $R_\mathfrak{p}/\mathfrak{p}R_\mathfrak{p}$.

\begin{corollary}\label{cor-Fibers-1}
    If $A$ is a smooth noetherian $B$-algebra, then
    \[
        \gd(A,B) \leq \sup_{\mathfrak{q} \in \spec B} \gd(A \otimes_B k(\mathfrak{q})).
    \]
\end{corollary}

\begin{proof}
    By Theorem \ref{Teo-Smo-Ida}, it suffices to provide an upper bound for $\pd_{\gr(A^e)} A$.  We start from the inequality in \eqref{Eq-cota-superior}, which gives
    \begin{equation}\label{Eq-pdsup}
         \pd_{\gr(A^e)} A \le \sup_{\mathfrak{p} \in \spec A} \rk_{A_\mathfrak{p}}(\Omega_{A|B})_\mathfrak{p},
    \end{equation} 
        where each $\rk_{A_\mathfrak{p}}(\Omega_{A|B})_\mathfrak{p}$  is defined as  $\rk_{k(\mathfrak{p})} (\Omega_{A|B} \otimes_A k(\mathfrak{p}))$.  

        Let $\phi:B\rightarrow A$. Fix $\mathfrak{p} \in \spec A$, and consider any prime ideal $\mathfrak{q} \subset B$ such that $\phi^{-1}(\mathfrak{p}) \subseteq \mathfrak{q}$. For instance, one may take $\mathfrak{q}=\phi^{-1}(\mathfrak{p})$. Then, after base change along $B \to k(\mathfrak{q})$, we have
    \[
        \Omega_{A|B} \otimes_A k(\mathfrak{p}) = \Omega_{A \otimes_B k(\mathfrak{q}) | k(\mathfrak{q})} \otimes_{A \otimes_B k(\mathfrak{q})} k(\mathfrak{p}).
    \]
    Since smoothness is preserved under base change, it follows that
    \[
        \rk_{k(\mathfrak{p})} (\Omega_{A|B} \otimes_A k(\mathfrak{p})) = \gd\big(A \otimes_B k(\mathfrak{q})\big)_{\mathfrak{p}'},
    \]
    where $\mathfrak{p}'=\mathfrak{p} \otimes_B k(\mathfrak{q})$; that is, the prime ideal of $A \otimes_B k(\mathfrak{q})$ corresponding to $\mathfrak{p}$.  

    Therefore, the supremum in \eqref{Eq-pdsup} can be rewritten as
    \begin{align*}
        \sup_{\mathfrak{p} \in \spec A} \rk_{A_\mathfrak{p}}(\Omega_{A|B})_\mathfrak{p} 
        &\leq \sup_{\mathfrak{q} \in \spec B} \sup_{\substack{\mathfrak{p} \in \spec A \\ \mathfrak{p} \cap B \subset \mathfrak{q}}} \gd\big(A \otimes_B k(\mathfrak{q})\big)_{\mathfrak{p}\otimes_B k(\mathfrak{q})} \\
        &= \sup_{\mathfrak{q} \in \spec B} \sup_{\substack{\mathfrak{p'} \in \spec(A \otimes_B k(\mathfrak{q}))}} \gd\big(A \otimes_B k(\mathfrak{q})\big)_{\mathfrak{p}'} \\   &= \sup_{\mathfrak{q} \in \spec B} \gd(A \otimes_B k(\mathfrak{q})),
    \end{align*}
    where the last equality follows from the fact that $A \otimes_B k(\mathfrak{q})$ is noetherian.
\end{proof}

\subsection{Smoothness of algebras with finite (co)homological dimension} \label{SecCohom}

\begin{theorem}\label{Thm-Cd-Smo}
Suppose $A$ is a noetherian flat $B$-algebra, essentially of finite type. If $\cdim(A,B)$ is finite, then $A$ is a smooth $B$-algebra.
\end{theorem}

\begin{proof}
Consider the same $(A^e,A)$-projective resolution of $A$, as an $A^e$-module, as in the proof of Theorem \ref{Teo-Smo-Ida} from the natural exact sequence given by the multiplication morphism. 
Using the associativity of the tensor product, there is an isomorphism
$$\Tor_i^{A^{e}}(A^{e} \otimes_A M,-)\cong\Tor_i^A(M,-), \quad i\geq 0, $$
for any $A$-module $M$. Now, the additivity of $\Tor$ implies that every $(A^e,A)$-projective module is $A$-flat, where the left $A$-module structure is induced by the map $a \mapsto a \otimes 1$. Consequently, $J$ and $K_i$ are also $A$-flat, being kernels of epimorphisms between flat modules. Since $\cdim(A,B)$ is finite, the module $K_n$ is $(A^{e},A)$-projective for some $n$, implying that $\fd_{A^{e}}A$ is also finite (in which $\fd_R M$ denotes the flat dimension of an $R$-module $M$).
As was shown in \cite{Rodi} if $\fd_{A^{e}} A < \infty$ then $B \rightarrow A$ has geometrically regular fibers. But, for $B$-algebras essentially of finite type, this is equivalent to being smooth by \cite[\href{https://stacks.math.columbia.edu/tag/038X}{Tag 038X}]{stacks-project}.
\end{proof}
Jointly with Theorem \ref{Teo-Smo-Ida} and Remark \ref{remSmoothCdim} we get the following
\begin{corollary}\label{Cor-EqCdim}   
    Let $A$ and $B$ be algebras as in Theorem \ref{Thm-Cd-Smo}. Then $A$ is smooth if and only if $\cdim(A,B)$ is finite.
\end{corollary}

Arguing in the same fashion as in the proofs of Theorems \ref{Teo-Smo-Ida} and \ref{Thm-Cd-Smo}, we can obtain another interesting consequence of the standard $(A^e,A)$-projective resolution of the $A^e$-module $A$.

\begin{corollary}\label{Cor-ReltAbs}     
     If $A$ is a flat $B$-algebra and $M$ is an $A^e$-module such that $\Tor_i^{(A^e,A)}(A,M)=0$ whenever $i\geq k$, for some integer $k$, then $\Tor_i^{A^e}(A,M)=0$ for all $i\geq k$. 
\end{corollary}

It is natural to inquire whether the finiteness of $\gd(A,B)$ implies that the $B$-algebra $A$ is smooth. Notably, there are well-known counterexamples in the case where $B=k$ and $k$ is a non-perfect field. Indeed, let $k=\mathbb{F}_p(t)$ be a transcendental extension of the finite field $\mathbb{F}_p$,  the $k$-algebra $A=k[x]/(x^p-t)$ has global dimension zero (being a field) but it is not a smooth $k$-algebra (because $A\otimes_k A$ is a local ring with nilpotent element $(x\otimes1-1\otimes x)$, hence has infinite global dimension). This underscores the need for additional conditions on $B$ to ensure the validity of the claim. Below we provide some sufficient conditions on $B$ under which the question holds true.

\begin{theorem}\label{Teo-Volta-Smot}
   Let $k$ be a perfect field, and let $B$ be a finitely generated $k$-algebra. Suppose that $A$ is a flat $B$-algebra, essentially of finite type. If $\gd(A,B)$ is finite, then $A$ is a smooth $B$-algebra. 
\end{theorem}

\begin{proof}

Using the fiberwise criterion of smoothness and the fact that smoothness is local on the target (meaning that it is sufficient to check it for closed points, see \cite[\href{https://stacks.math.columbia.edu/tag/02G1}{Tag 02G1}]{stacks-project} and \cite[Proposition 6.15]{Gor}), we prove that:
\begin{itemize}
\item[$(1)$] $\gd(A\otimes_B k(\mathfrak{m}))$ is finite,
\item[$(2)$] $k(\mathfrak{m})$ is a perfect field,
\end{itemize}
for every maximal ideal $\mathfrak{m}$ in $B$.
Item $(2)$ is a direct consequence of the general Nullstellensatz theorem as stated in \cite[Theorem 5.6.7]{Found-Comm}, and the fact that algebraic extensions of a perfect field are also perfect. To prove $(1)$, we consider an $A\otimes k(\mathfrak{m})$-module $M$ treated as an $A$-module, and take the standard $(A,B)$-projective resolution of $M$:
    
    \centerline{
    \xymatrix{
												&&&&K_1 \ar@{^{(}->}[rd]		\\		
	...	\ar[rr] \ar[rd]		&&A\otimes_B K_{n-1}  \ar[r]	 &...	\ar[rr]  \ar@{.>}[ru]	 &&A\otimes_B K_0 \ar@{->>}[rd] \ar[rr] 	&&A \otimes_B M \ar@{->>}[r]	&M, \\
		        &K_n  \ar@{^{(}->}[ru]								&&&&&K_0  \ar@{^{(}->}[ru]
    }}

As $\gd(A,B)$ is finite, then $K_n$ is $(A,B)$-projective for some $n$. Note that $A\otimes_B M$, $K_i$, and $A\otimes_B K_i$ all have structures as $A\otimes_B k(\mathfrak{m})$-modules. Furthermore, since all of these $A\otimes_B k(\mathfrak{m})$-modules can be viewed as $k(\mathfrak{m})$-vector spaces, we conclude that $A\otimes_B M$ and $A\otimes_B K_i$ are $A\otimes_B k(\mathfrak{m})$-free modules. The same argument shows that $K_n$ is an $A\otimes k(\mathfrak{m})$-projective module. This proves that $\gd(A\otimes_B k(\mathfrak{m}))$ is finite.

To complete the proof, we utilize the fact that for perfect fields $k(\mathfrak{m})$, being geometrically regular over $k(\mathfrak{m})$ is equivalent to being regular over $k(\mathfrak{m})$ by \cite[Chp X, Section 6.4]{Bour}. Additionally, being geometrically regular over a field is equivalent to the smoothness of the fiber over $k(m)$, as per \cite[\href{https://stacks.math.columbia.edu/tag/038X}{Tag 038X}]{stacks-project}.
\end{proof}

Finally, the previous theorem, together with Theorem \ref{Teo-Smo-Ida}, yield the following result.
\begin{corollary}\label{CorCriteriaPerfect}
Let $k$ be a perfect field, $B$ a finitely generated $k$-algebra, and $A$ a flat $B$-algebra,  essentially of finite type. Then $A$ is smooth if and only if $\gd(A,B)$ is finite.
\end{corollary}

\begin{corollary}\label{cor-Fibers-2}
    Let $A$, $B$ be algebras as in Theorem \ref{Teo-Volta-Smot}. Then, for every prime ideal $\mathfrak{q} \subset B$,     \begin{equation}\label{eq_relativeGl_fibers_inf}        
        \gd(A,B) \geq \gd(A \otimes_B k(\mathfrak{q})).
    \end{equation}
    Furthermore, jointly with Corollary \ref{cor-Fibers-1} we get that 
    \begin{align}
         \gd(A,B) &= \sup_{\mathfrak{q} \in \spec B} \gd(A \otimes_B k(\mathfrak{q})) \label{eq_relativeGl_fibers}\\
         &= \sup_{\mathfrak{m} \in \Max B} \gd(A \otimes_B k(\mathfrak{m})) \nonumber.
    \end{align} 
\end{corollary}

\begin{proof}
Indeed, the first claim follows directly from the proof of item~(1) in Theorem~\ref{Teo-Volta-Smot}.  
If the extension $B \subseteq A$ is smooth, the second claim follows from the first one together with Corollary~\ref{cor-Fibers-1}.  
In the non-smooth case, both sides of~\eqref{eq_relativeGl_fibers} are infinite, which completes the proof.
\end{proof}

\begin{rmk} Observe that inequality \eqref{eq_relativeGl_fibers_inf} remains valid even without the assumption that $A$ is $B$-flat, since this hypothesis is not required in the proof of item (1) of Theorem \ref{Teo-Volta-Smot}.
\end{rmk}
 
This corollary shows that the relative global dimension of an extension $B \subseteq A$ measures the relative dimension (in the smooth case) of this extension. Moreover, it allows us to establish a (relative) analogue of \cite[Theorem~16]{Aus55}. Namely, we have the following result.

\begin{proposition}\label{prop-Gldim-tensor}
Let $B \subseteq A$ and $D \subseteq C$ be two extensions satisfying the conditions of Theorem~\ref{Teo-Volta-Smot}. Then
\[
    \gd(A \otimes_k C, B \otimes_k D)
    = \gd(A, B) + \gd(C, D).
\]
\end{proposition}

\begin{proof}
Since both extensions $B \subseteq A$ and $D \subseteq C$ satisfy the hypotheses of Theorem~\ref{Teo-Volta-Smot}, the tensor product extension
\[
    B \otimes_k D \;\subseteq\; A \otimes_k C
\]
also satisfies these hypotheses, as flatness and being essentially of finite type are preserved under both base change and composition.

By Corollary~\ref{cor-Fibers-2}, we have
\[
    \gd(A \otimes_k C, B \otimes_k D)
    = \sup_{\mathfrak{p} \in \Max B \otimes_k D}
      \gd\big( (A \otimes_k C) \otimes_{B \otimes_k D} k(\mathfrak{p}) \big).
\]

We can use the description in \cite[\href{https://stacks.math.columbia.edu/tag/01JT}{Tag 01JT}]{stacks-project} of points on the fiber product of schemes to describe the maximal ideals of $B \otimes_k D$. These correspond to triples $(\mathfrak{q}, \mathfrak{r}, \mathfrak{s})$, where $\mathfrak{q}$ is a maximal ideal of $B$, $\mathfrak{r}$ is a maximal ideal of $D$, and $\mathfrak{s}$ is a maximal ideal of $k(\mathfrak{q}) \otimes_k k(\mathfrak{r})$. Since $k$ is perfect, both $k(\mathfrak{q})$ and $k(\mathfrak{r})$ are finite separable extensions of $k$, and their tensor product can be decomposed uniquely, up to isomorphism, as a product of finite separable extensions $L_i$ of $k$. The fields $L_i$ are precisely the residue fields of the maximal ideals of $B \otimes_k D$ corresponding to the pair $(\mathfrak{q}, \mathfrak{r}) \in \Max B \times \Max D$.

Moreover, such decomposition implies the isomorphisms:
\begin{align}
    (A\otimes_B k(\mathfrak{q}))\otimes_k (C\otimes_D k(\mathfrak{r}))
    &\simeq (A\otimes_k C) \otimes_{B\otimes_k D} (k(\mathfrak{q})\otimes_k k(\mathfrak{r})) \label{eq-tensor}\\ 
    &\simeq \bigtimes_{i=1}^l (A\otimes_k C)\otimes_{B\otimes_k D} L_i \nonumber
\end{align}

To conclude the argument, we use the fact that, for noetherian rings, the global dimension can be verified locally, and that for local noetherian algebras over a perfect field, the global dimension coincides with the projective dimension as a bimodule. Consequently, the dimension considered in \cite[Chp.~IX, Section~7]{C-Eil} agrees with the global dimension, allowing us to apply \cite[Chp.~IX, Propositions~7.3 and~7.4]{C-Eil} to~\eqref{eq-tensor}.

Hence, the global dimension of the fibers of the extension $B \otimes_k D \subseteq A \otimes_k C$ is bounded by the sum of the global dimensions of the fibers of $B \subseteq A$ and $D \subseteq C$. Moreover, equality is attained for a suitable choice of $L_i$. Taking the supremum over all maximal ideals of $B \otimes_k D$ then yields the desired equality.

\end{proof}

\section{Examples and final remarks} \label{SecExamp}

\subsection{Examples}

\begin{example}
    Let $X$ and $Y$ be algebraic varieties, and let 
    $\pi_Y : X \times Y \to Y$ be the canonical projection. 
    In this case, the corresponding inclusion map 
    $$\pi_Y^* : k[Y] \to k[X \times Y] = k[X] \otimes_k k[Y]$$ 
    is a smooth morphism if and only if $X$ is a smooth variety. Therefore, by Theorem~\ref{Teo-Smo-Ida},
    $\gd(k[X] \otimes_k k[Y], k[Y]) < \infty$ if and only if $\gd(k[X])<\infty$.
    Moreover, using Corollary \ref{cor-Fibers-2} one gets:
    \[
        \gd(k[X] \otimes_k k[Y], k[Y]) = \gd k[X] .
    \]
\end{example}

\begin{example} \label{Ex-etale}
    Let $A$ be a smooth noetherian $B$-algebra such that $\Omega_{A|B} = 0$. In this case, the extension $B \subseteq A$ is called \emph{\'{e}tale}. Then $\gr(A^e) = A$, and Theorem \ref{Teo-Smo-Ida} implies that $\gd(A,B) = 0$. Moreover, for each prime ideal $\mathfrak{p} \in \spec(B)$, we can compute the global dimension of $A \otimes_B k(\mathfrak{p})$, which is bounded above by $\gd(A,B) = 0$, hence $\gd(A \otimes_B k(\mathfrak{p})) = 0$. A simple example of an \'{e}tale extension is given by any localization $S^{-1}A$ for a multiplicative subset $S$ of a domain $A$.
\end{example}

The following construction provides an example of a smooth noetherian extension $B \subseteq A$ where the $A$-module $A \otimes_B A$ fails to be projective.

\begin{example}\label{Ex-Aenonproj}    
  Let $B = k[t]$ and $A = k[t,t^{-1}] \times k[x] = A_1 \times A_2$, where $A$ is endowed with the diagonal $B$-algebra structure. Then this extension is smooth, since both $A_1$ and $A_2$ are smooth $B$-algebras (the latter trivially, and the former by Example~\ref{Ex-etale}), and smoothness is local on the source and the target hence preserved under finite products.  Moreover, we can identify  the $A$-module $A \otimes_B A$ as the direct sum
    \[
    A \otimes_B A = (k[t,t^{-1}] \oplus k[t,t^{-1}]) \oplus (k[t,t^{-1}] \oplus k[t]).
    \]
       Via the decomposition $A = A_1 \times A_2$, this induces $A_i$-module structures  
    \[
M_1 := k[t,t^{-1}] \oplus k[t,t^{-1}] \quad \text{over } A_1, 
\qquad
M_2 := k[t,t^{-1}] \oplus k[t] \quad \text{over } A_2,
    \] 
    with $A_2M_1=0$ and $A_1M_2=0$. It follows that $A \otimes_B A$ is projective as an $A$-module if and only if $M_1$ is projective over $A_1$ and $M_2$ is projective over $A_2$. Since $M_2$ is not projective as an $A_2$-module, we conclude that $A \otimes_B A$ is not projective as an $A$-module.
\end{example}

Using Corollary~\ref{CorCriteriaPerfect}, one can easily construct 
examples of extensions with infinite relative global dimension, 
even when both $A$ and $B$ are smooth $k$-algebras.

\begin{example}
Let $A = k[x_1, \dots, x_n]$ and let $G \subset \mathrm{GL}_n(k)$ be a finite group generated by pseudoreflections acting linearly on $A$.  Although $A^G$ is a polynomial ring (by the Chevalley--Shephard--Todd theorem), 
the inclusion $A^G \hookrightarrow A$ is \emph{not smooth} (unless the action of $G$ is trivial), because each pseudoreflection 
fixes a hyperplane, causing ramification along these codimension-one loci where 
the module of differentials $\Omega_{A|A^G}$ fails to be locally free. Consequently, 
\[
       \gd(A, A^G) = \infty.
\]
For instance, in the one-dimensional case $A = k[x]$ and $G = \{1, -1\}$, we get
\[
       \gd(k[x], k[x^2]) =\infty.
\]
\end{example}

\subsection{Smoothness and (relative) homology}

The Hochschild homology of a $k$-algebra $A$ is defined as
\[
    \HH_*(A) = \Tor_*^{A \otimes_k A}(A, A).
\]
By the Hochschild–Kostant–Rosenberg theorem (\cite[Theorem~5.2]{HKR-62}), the Hochschild homology of a smooth $k$-algebra vanishes in sufficiently high degrees.  
The converse was established independently in \cite{Av-VP} and \cite{BACH}, the latter under the additional assumption that $\mathrm{char}\, k = 0$.  
An analogous question arises in the noncommutative setting, where it is known as Han’s conjecture; see \cite{CLMS, CLMS21, IM21} and the references therein for recent progress in this direction.

It is natural to ask a similar question in the relative setting.  
Given an extension of commutative algebras $B \subseteq A$ and an $A$-bimodule $M$, Hochschild~\cite{Hoc} introduced the \emph{relative Hochschild homology} of the extension $B \subseteq A$, defined in terms of the corresponding relative bar complex:
\[
    \HH_*(A|B) = \Tor_*^{(A \otimes_k A, B \otimes_k B )}(A, A).
\]

We prove 

\begin{theorem}
    Let $k$ be a perfect field, $B$ a finitely generated $k$-algebra and $A$ a flat $B$-algebra, essentially of finite type. Then the following are equivalent:
    \begin{enumerate}
        \item[(i)] $B\subseteq A$ is smooth;
        \item[(ii)] $\gd(A,B)<\infty$;
        \item[(iii)] $\HH_j(A|B)=0$, for $j$ sufficiently large.        
    \end{enumerate}
\end{theorem}

\begin{proof}
The implication $[(i) \Rightarrow (ii)]$ follows from Theorem~\ref{Teo-Smo-Ida}.  

Assuming $(ii)$ and taking $C = A$ and $D = B$, by Proposition~\ref{prop-Gldim-tensor} we obtain
\[
    \gd(A \otimes_k A, B \otimes_k B) = 2\gd(A,B) < \infty,
\]
which proves $(iii)$.

To show $[(iii) \Rightarrow (i)]$, observe that
\[
    (A \otimes_B \cdots \otimes_B A) \otimes_{A \otimes_k A} A 
    \simeq 
    (A \otimes_B \cdots \otimes_B A) \otimes_{A \otimes_B A} A.
\]
Together with the fact that the relative bar resolution from~\cite[Section~3]{Hoc} is simultaneously an $(A \otimes_B A, A)$-projective resolution and an $(A \otimes_k A, B \otimes_k B)$-projective resolution, this yields
\[
    \Tor_j^{(A \otimes_k A, B \otimes_k B)}(A, A)
    \simeq 
    \Tor_j^{(A \otimes_B A, A)}(A, A).
\]
Hence, applying Corollary~\ref{Cor-ReltAbs}, we deduce from condition~$(i)$ that $\Tor_j^{A^e}(A, A) = 0$ for all sufficiently large~$j$.  
Finally, by the Semi-Rigidity Theorem of~\cite{Av-Iyen}, this vanishing implies the smoothness of the extension $B \subseteq A$.
\end{proof}

\end{document}